\documentclass[12pt]{amsart}
\usepackage{amscd,graphicx,accents,stmaryrd}
\calclayout
\numberwithin{equation}{section} 

\newcommand{\ud}{\,d} 
\newcommand{\Sym}{\mathbb{S}} 
\newcommand{\R}{\mathbb{R}}
\renewcommand\P{{\mathcal P}}
\newcommand\T{{\mathcal T}}
\newcommand\x{\times}
\renewcommand{\div}{\operatorname{div}} 
\newcommand{\grad}{\operatorname{grad}} 
\newcommand{\eps}{\operatorname{\epsilon}}
\newtheorem{thm}{Theorem}[section] 
 
\newtheorem{lemma}[thm]{Lemma}

\newtheorem{rem}[thm]{Remark}

\newcommand{\0}{{\scriptscriptstyle0}}
\newcommand{\jump}[1]{\llbracket #1 \rrbracket}
  
\begin{document}
\title{Nonconforming tetrahedral mixed finite elements for elasticity}

\thanks{The work of the first author was supported by NSF grant
DMS-1115291. The work of the second author was supported by
NSF grant DMS-0811052 and the Sloan Foundation.  The work of the
third author was supported by the Research Council of Norway
through a Centre of Excellence grant
to the Centre of Mathematics for Applications.}

\author{Douglas Arnold, Gerard Awanou and Ragnar Winther}

\address{School of Mathematics, University of Minnesota, Minneapolis,
Minnesota 55455}
\email{arnold@umn.edu}
\urladdr{http://umn.edu/\~{}arnold}

\address{Department of Mathematics, Statistics, and Computer Science, M/C
249,
University of Illinois at Chicago,
Chicago, IL 60607-7045, USA}
\email{awanou@uic.edu}
\urladdr{http://www.math.uic.edu/\~{}awanou}

\address{Centre of Mathematics for Applications and Department of
Informatics, University of Oslo, P.O. Box 1053, Blindern, 0316 Oslo, Norway}
\email{ragnar.winther@cma.uio.no}
\urladdr{http://folk.uio.no/\char'176rwinther}



\begin{abstract}
This paper presents a nonconforming finite element
approximation of  the space of symmetric tensors
with square integrable divergence, on tetrahedral meshes.  Used for
stress approximation together with 
the full space of piecewise linear vector fields for displacement, this
gives a stable mixed finite element method which is shown to be linearly
convergent for both the stress and displacement, and which is significantly
simpler than any stable conforming mixed finite element method.  The
method may be viewed as the three-dimensional analogue
of a previously developed element in two dimensions.  As in that case, a
variant of the method is proposed as well, in which the displacement
approximation is reduced to piecewise rigid motions and the stress space
is reduced accordingly, but the linear convergence is retained.
\end{abstract}

\maketitle

\section{Introduction}
Mixed finite element methods for elasticity simultaneously approximate the displacement
vector field and the stress tensor field.  Conforming methods based on the classical
Hellinger--Reissner variational formulation require a finite element space for the
stress tensor that is contained in $H(\div,\Omega;\Sym)$, the space of symmetric $n\times n$ tensor
fields which are square integrable with square integrable divergence.
For a stable method, this stress space must be compatible with the finite element space used for the displacement,
which is a subspace of the vector-valued $L^2$ function space.  It has proven difficult to devise
such pairs of spaces.  While some stable pairs have been successfully constructed in both $2$
and $3$ dimensions, the resulting elements tend to be quite complicated, especially
in $3$ dimensions.  For this reason, much attention has been paid to constructing elements
which fulfill desired stability, consistency, and convergence conditions, but
which relax the requirement that the stress space be contained in $H(\div,\Omega;\Sym)$ in
one of two ways: either by relaxing the interelement continuity requirements, which leads
to \emph{nonconforming} mixed finite elements, or
by relaxing the symmetry requirement, which leads to mixed finite elements with \emph{weak symmetry}.
In this paper we construct a new nonconforming mixed finite element for elasticity
in three dimensions based on tetrahedral meshes, analogous to
a two-dimensional element defined in \cite{Arnold2003}.
The space $\Sigma_K$ of shape functions
on a tetrahedral element $K$ (which is defined in \eqref{defS} below)
is a subspace of the space $\P_2(K;\Sym)$, the space of symmetric
tensors with components which are polynomials of degree at most $2$.
It contains $\P_1(K;\Sym)$ and has dimension 42.  The degrees of freedom for $\sigma\in\Sigma_K$
are the integral of $\sigma$ over $K$ (this is six degrees of freedom, since $\sigma$
has six components), and the integral and linear moments of $\sigma n$ on each face of $K$
(nine degrees of freedom per face).  For the displacements we simply take
$\P_1(K,\R^3)$ as the shape functions and use only interior degrees of freedom so
as not to impose any interelement degrees of freedom. See the element diagrams
in Figure~\ref{fg:dofs}.  We note that, since there are no degrees of freedom
associated to vertices or edges, only to faces and the interior, our elements may
be implemented through hybridization, which may simplify the implementation.
See \cite{minc} for the general idea, or \cite{Guzman2010} for a case close to
the present one.

\begin{figure}[ht]
\centerline{\includegraphics[width=40mm]{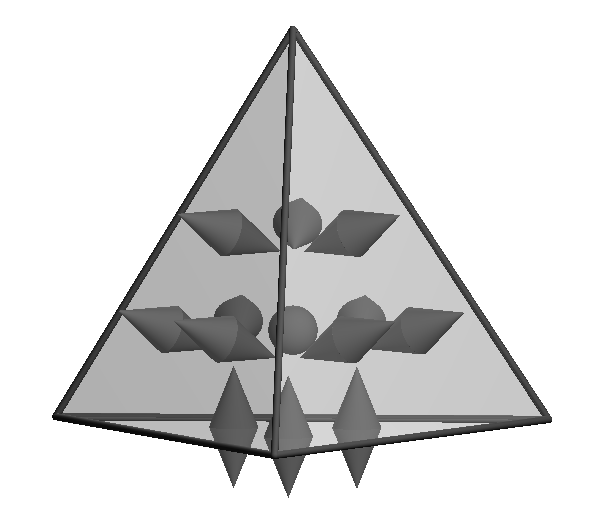}\quad
\includegraphics[width=40mm]{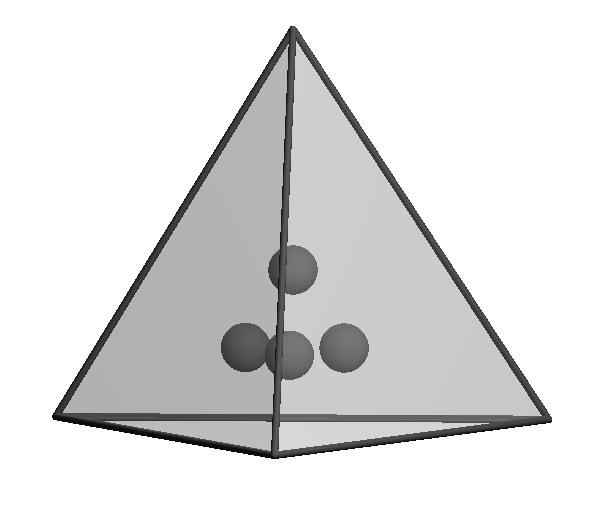}}
\caption{\label{fg:dofs}Degrees of freedom for the stress $\sigma$ (left) and displacement $u$
(right).
The arrows represent moments of $\sigma n$, which has three components,
and so there are $9$ degrees of freedom associated to each face.  The interior
degrees of freedom are the integrals of $\sigma$ and $u$, which have 6 and 3 components,
respectively.}
\end{figure}

After some preliminaries in section~\ref{prelim},
in section~\ref{elts} we define the shape function space $\Sigma_K$ and
prove unisolvence of the degrees of freedom.  In section~\ref{errors} we establish
the stability, consistency, and
convergence of the resulting mixed method.  Finally
in section~\ref{reduced} we describe a variant of the method which reduces the displacement
space to the space of piecewise rigid motions and reduces the stress space accordingly.
The results of this paper were announced in \cite{Awanou10b}.

As mentioned, conforming mixed finite elements for elasticity tend to be quite complicated.
The earliest elements, which worked only in two dimensions, used composite elements
for stress \cite{Johnson1978,Arnold1984}.  Much more recently, elements using
polynomial shape functions were developed for simplicial meshes in two \cite{Arnold2002}
and three dimensions \cite{Adams2005,Arnold2008} and for rectangular meshes \cite{Arnold2005,Chen2010}.
Heuristics given in \cite{Arnold2002} and \cite{Arnold2008} 
indicate that it is not possible to construct significantly
simpler elements with polynomial shape functions and which preserve
both the conformity and symmetry of the stress.  Many authors have
developed mixed elements with weak symmetry
\cite{Amara1979,Arnold1984a,Stenberg1986,Stenberg1988,Stenberg1988a,Morley1989,Arnold2006,Arnold2007,Falk08,Guzman10c,Guzman10d,Guzman2012},
which we will not pursue here.
For nonconforming methods with strong symmetry, which is the subject of this
paper, there have been several elements proposed for rectangular meshes
\cite{Yi2005,Yi2006,Hu2007/08,Awanou2009,Man2009}, but very little work
on simplicial meshes.  A two-dimensional nonconforming element of low degree
was developed by two of the present authors in \cite{Arnold2003}.  As shape functions for
stress it uses a $15$ dimensional subspace of the space of all quadratic symmetric tensors,
while for the displacement it uses piecewise linear vector fields.
A second element was also introduced in \cite{Arnold2003}, for which the
stress shape function space was reduced to dimension $12$ and the displacement
functions reduced to the piecewise rigid motions.
In \cite{Guzman2010} Gopalakrishnan and Guzm\'an developed
a family of simplicial elements,
in both two and three dimensions.  As shape functions they used the space of
all symmetric tensors of polynomial degree at most $k+1$, paired with
piecewise polynomial vector fields of dimension $k$, for $k\ge 1$.
Thus, in two dimensions and in the lowest degree case, they use an $18$ dimensional
space of shape functions for stress, while in three dimensions, the space
has dimension $60$.  Gopalakrishnan and Guzm\'an also proposed a reduced
variant of their space, in which the displacement space remains the full
space of piecewise polynomials of degree $k$, but the dimension of the
stress space is reduced to $15$ in two dimensions and to $42$ in three dimensions.
However, their reduced spaces have a drawback, in that they are not uniquely
defined, but for each edge of the triangulation require a choice of
a favored endpoint of the edge.  In particular, in two dimensions, the reduced space
of \cite{Guzman2010} uses the same displacement space as the non-reduced space
of \cite{Arnold2003}, uses a stress space of the same dimension, and uses
identical degrees of freedom, but the two spaces do \emph{not} coincide (since
the space of \cite{Arnold2003} does not require a choice of favored edge endpoints).

The elements introduced here may be regarded as the three-dimensional analogue of the
element in \cite{Arnold2003}.  Again, they have the same displacement space and the
same degrees of freedom as the reduced three-dimensional elements of \cite{Guzman2010},
but the stress spaces do not coincide.  Also, as in the two-dimensional case, our reduced
space is of lower dimension than any that has been heretofore proposed.

\section{Preliminaries} \label{prelim}
Let $\Omega\subset\R^3$ be a bounded domain.  We denote by $\Sym$ 
the space of $3 \times 3$ symmetric matrices and by $L^2(\Omega;\R^3)$
and $L^2(\Omega;\Sym)$ the space of square-integrable vector fields
and symmetric matrix fields on $\Omega$, respectively.  The space $H(\div,\Omega;\Sym)$
consists of matrix fields $\tau\in L^2(\Omega;\Sym)$ with row-wise divergence,
$\div\tau$, in $L^2(\Omega;\R^3)$.
The Hellinger--Reissner variational formulation 
seeks $(\sigma,u) \in H(\div,\Omega;\Sym) \times L^2(\Omega; \R^3)$ 
such that
\begin{equation}\label{ncf}
\begin{aligned}
\int_{\Omega} (A \sigma : \tau + \div \tau \cdot u) \ud x & = 0, &&\tau \in H(\div,\Omega;\Sym) \\
\int_{\Omega} \div \sigma \cdot v \ud x & = \int_{\Omega} f \cdot v \ud x,  &&v \in L^2(\Omega; 
\R^n).
\end{aligned}
\end{equation}
Here $\sigma:\tau$ denotes the Frobenius inner products of matrices $\sigma$ and $\tau$,
and $A=A(x): \Sym \to \Sym$ denotes the compliance tensor, a linear operator which
is bounded and symmetric
positive definite uniformly for $x \in \Omega$.  The solution $u$ solves the Dirichlet
problem for the Lam\'e equations and so belongs to $\ring H^1(\Omega;\R^3)$.  If the
domain $\Omega$ is smooth and the compliance tensor $A$ is
smooth, then $(\sigma,u)\in H^1(\Omega;\Sym) \times H^2(\Omega;\R^3)$ and 
\begin{equation}\label{regularity}
\| \sigma \|_1 + \| u \|_2 \le c \|f \|_0.
\end{equation}
with a constant $c$ depending on $\Omega$ and $A$.  The same regularity holds
if the domain is a convex polyhedron, at least in the isotropic homogeneous
case.  See \cite{Nicaise}.

We shall also use spaces of the form $H^k(\Omega;X)$ where $X$ is a finite-dimensional
vector space and $k$ is a nonnegative integer, the Sobolev space of functions
$\Omega\to X$ for which all derivatives of order at most $k$ are
square integrable.  The norm is denoted by $\|\cdot\|_{\Omega,k}$ or $\|\cdot\|_k$.

To discretize \eqref{ncf}, we choose finite-dimensional
subspaces $\Sigma_h\subset L^2(\Omega;\Sym)$ and $V_h\subset L^2(\Omega;\R^3)$.  Assuming
that $\Sigma_h$ consists of matrix fields which are piecewise polynomial with respect to some
mesh $\T_h$ of $\Omega$, we define $\div_h\tau\in L^2(\Omega;\R^3)$
by applying the (row-wise) divergence operator piecewise.  
A mixed finite element approximation of \eqref{ncf} is then obtained by seeking $(\sigma_h,u_h) \in \Sigma_h 
\times V_h$ such that:
\begin{equation}\label{ncfD}
\begin{aligned} 
\int_{\Omega} (A \sigma_h : \tau + \div_h \tau  \cdot u_h )\ud x & = 0, 
&&\tau \in \Sigma_h \\
\int_{\Omega} \div_h \sigma_h \cdot v\ud x & = \int_{\Omega} f \cdot v_h \ud x, &&v \in V_h.
\end{aligned}
\end{equation}
If $\Sigma_h\subset H(\div,\Omega;\Sym)$ this is a conforming method, otherwise, as for the elements
developed below, it is nonconforming.  We recall that a piecewise smooth matrix field $\tau$ belongs
to $H(\div)$ if and only if whenever two tetrahedra in $\T_h$ meet in a common face,
the jump $\jump{\tau n}$ of the normal components $\tau n$ across the face vanish.

\section{Definition of the new elements} \label{elts}
We define the finite element spaces $\Sigma_h$ and $V_h$ in the usual way, by
specifying spaces of shape functions and degrees of freedom.  The space $V_h$ 
is simply the space of all piecewise linear vector fields with respect
to the given tetrahedral mesh $\T_h$ of $\Omega$ (which we therefore assume is
polyhedral).  Thus the shape function space
on an element $K\in\T_h$ is simply $V_K=\P_1(K;\R^3)$, the space of polynomial vector fields
on $K$ of degree at most $1$.  For degrees of freedom we choose the moments
$v\mapsto \int_K v\cdot w\ud x$ with weights $w\in V_K$.  Since no degrees of freedom
are associated with the proper subsimplices of $K$, no interelement continuity
is imposed on $V_h$.  The associated
projection $P_h:L^2(\Omega;\R^3)\to V_h$ is the $L^2$ projection.

To define the space $\Sigma_h$ we introduce some notation.  If $u$ is a unit vector,
let $Q_u:\R^3\to u^\perp$ be the orthogonal projection onto the plane orthogonal
to $u$.  Then $Q_u$ is given by the symmetric matrix $I-uu'$.
For a tetrahedron $K$, let $\Delta_k(K)$
denote the subsimplices of dimension $k$ (vertices, edges, faces, and
tetrahedra) of $K$.  For an edge $e\in \Delta_1(K)$ let $s_e$ be a unit vector parallel
to $e$ and, for a face $f\in \Delta_2(K)$, let $n_f$ be its outward unit normal.
We can then define the shape function space
\begin{align}\label{defS}
\Sigma_K
&= \{\, \sigma \in \mathcal{P}_2(K;\mathbb{S}) \ | \
Q_{s_e} \sigma Q_{s_e}|_e \in \P_1(e;\Sym)\,\forall e \in \Delta_1(K) \,\}.
\end{align}
For $\sigma\in\mathcal{P}_2(K;\mathbb{S})$, $Q_{s_e} \sigma Q_{s_e}|_e$ is a quadratic polynomial on $e$ taking
values in the $3$-dimensional subspace $Q_{s_e}\Sym Q_{s_e}$ of $\Sym$.  As illustration, for $s_e=(0,0,1)'$ and $\sigma=(\sigma_{ij})_{i,j=1,\ldots,3} \in \mathbb{S}$, we have
$$
Q_{s_e} \sigma Q_{s_e} = \begin{pmatrix}\sigma_{11} & \sigma_{12} & 0 \\ \sigma_{12} & \sigma_{22} & 0 \\ 0 & 0 & 0\end{pmatrix}.
$$
Thus the requirement
that $Q_{s_e} \sigma Q_{s_e}|_e$ belong to $\P_1$ represents $3$ linear constraints on $\sigma$, and so
$\dim \Sigma_K \geq 60 - 3 \x 6 = 42$.
We shall now specify $42$ degrees of freedom (linear functionals) and show unisolvence, i.e., that if all the degrees of freedom
vanish for some $\sigma\in \Sigma_K$, then $\sigma$ vanishes.  This will imply that $\dim\Sigma_K\le 42$,
and so the dimension is exactly $42$.

The degrees of freedom we take are:
\begin{align}\label{dof1}
&\int_f \sigma n_f\cdot v\ud s, \quad v\in \P_1(f;\R^3),\ f\in\Delta_2(K),  &&\text{($36$ degrees of freedom)},
\\
\label{dof2}
&\int_K \sigma\ud x,  &&\text{($6$ degrees of freedom)}.
\end{align}

The following lemma will be used in the proof of unisolvence.
\begin{lemma} \label{lemT}
Let $f_i$ and $f_j$ be the faces of $K$ opposite two distinct vertices $v_i$ and $v_j$
and let $e$ be their common edge, with endpoints $v_k$ and $v_l$.  Given $\beta,\gamma\in\R$,
there exists a unique $p\in\P_2(K)$ satisfying the following four conditions (see Figure~\ref{fg:tetlemma}):
\begin{enumerate}
 \item $p|_e \in \P_1(e)$,
 \item $p(v_k)=\beta$, \ $p(v_l)=\gamma$,
 \item $p|_{f_i}\perp_{L^2} \P_1(f_i)$, $p_{f_j}\perp_{L^2} \P_1(f_j)$,
 \item $\int_K p\ud x=0$.
\end{enumerate}
Moreover $p(v_i)=p(v_j)=3(\beta+\gamma)/2$.
\end{lemma}
\begin{figure}[ht]
\centerline{\includegraphics[width=70mm]{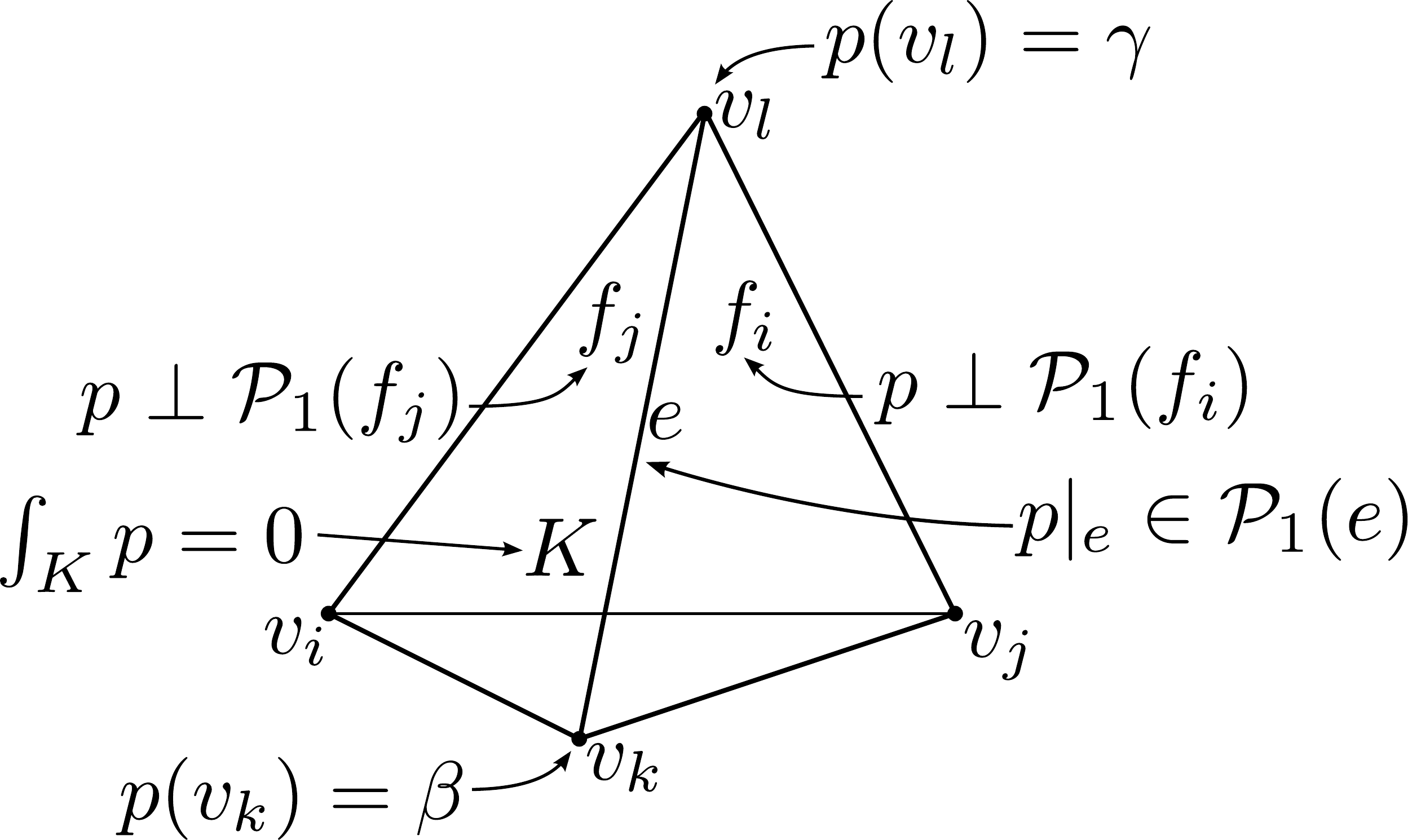}}
\caption{\label{fg:tetlemma}The conditions of Lemma~\ref{lemT}.}
\end{figure}

\begin{proof}
 For uniqueness we must show that if $p\in\P_2(K)$ satisfies (1)--(4) with $\beta=\gamma=0$, then $p$ vanishes.
Certainly, from (1) and (2), $p$ vanishes on $e$, and then, using (3), $p$ vanishes on $f_i$ and $f_j$.
Therefore $p=c\lambda_i\lambda_j$ where $\lambda_i\in\P_1(K)$ is the barycentric coordinate
function equal to $0$ on $f_i$ and $1$ at $v_i$, similarly for $\lambda_j$, and $c$ is a constant.
Integrating this equation over $K$ and invoking (4) we conclude that $p$ does indeed vanish.

To show the existence of $p\in\P_2(K)$, we simply exhibit its formula
in terms of barycentric coordinates:
\begin{multline*}
p  = \beta \lambda_k^2 + (\beta+\gamma)\lambda_k\lambda_l
+ \gamma \lambda_l^2 + \frac{3}{2}(\beta+\gamma)( \lambda_i^2+\lambda_j^2) \\
+(-5 \beta - \gamma) (\lambda_i+\lambda_j)\lambda_k + (-\beta-5\gamma)
(\lambda_i+\lambda_j) \lambda_l + 3(\beta+\gamma)\lambda_i \lambda_j.
\end{multline*}
That this function satisfies (1)--(4)
follows from the elementary formula
\begin{equation*}
 \int_T \lambda^\alpha = \frac{\alpha_1!\cdots\alpha_{d+1}!d!}{(|\alpha|+d)!}|T|, \quad \alpha\in\mathbb N_0^{d+1},
\end{equation*}
for the
integral of a barycentric monomial over a simplex $T$ of dimension $d$,
which can be established by induction (see, e.g., \cite{Eisenberg-Malvern}).
\end{proof}

We are now ready to prove the claimed unisolvence result.
\begin{thm} \label{HdivU}
The degrees of freedom given by \eqref{dof1} and \eqref{dof2}
are unisolvent for the shape function space $\Sigma_K$ defined by \eqref{defS}:
if the degrees of freedom all vanish for some $\sigma\in\Sigma_K$, then $\sigma=0$.
\end{thm}
\begin{proof}
Let $g_i=\grad \lambda_i$ be the gradient of the $i$th barycentric
coordinate function.  Thus $g_i$ is an inward normal vector to
the face $f_i$ with length $1/h_i$ where $h_i$ is the distance
from the $i$th vertex to $f_i$.  Note that any three of the $g_i$
form a basis for $\R^3$ and that $\sum_i g_i=0$.

For $\sigma\in\Sigma_K$, define $\sigma_{ij}=\sigma_{ji}=
g_i'\sigma g_j\in\P_2(K)$.  We shall show that if
$\sigma\in\Sigma_K$ and all the degrees of freedom vanish,
then $\sigma_{ij}\equiv0$ on $K$ for all $i\ne j$.  This is
sufficient, since, fixing $j$ and varying $i$, we conclude that
$\sigma g_j\equiv0$, and, then, since this holds for each $j$,
that $\sigma\equiv 0$.

If $e$
is an edge of the faces $f_i$ and $f_j$ of $K$, which may or may
not coincide, then $\sigma_{ij}=g_i'\sigma g_j=g_i'Q_s\sigma Q_s g_j$.
Thus, from the definition \eqref{defS} of the space $\Sigma_K$, 
$\sigma_{ij}$ is linear on $e$.  In particular,
$\sigma_{ii}$ is linear on each edge of $f_i$.  Thus $p:=\sigma_{ii}|_{f_i}$ is a quadratic polynomial on
$f_i$ whose restriction to each edge of $f_i$ is linear.  Therefore, on the boundary
of $f_i$, $p$ coincides with its linear interpolant, and, since
a quadratic function on a triangle is determined by its boundary values,
$p$ is linear.  Thus $\sigma_{ii}$ is actually a linear polynomial on $f_i$, and,
in view of the degrees of freedom \eqref{dof1}, we conclude that
$\sigma_{ii}$ vanishes on $f_i$.

For any pair $(l,k)$ of distinct indices (that is, $1\le
l,k\le 4$ and $l\ne k$), define
\begin{equation}\label{k}
\beta_{lk} = \sigma_{ij}(v_k), \quad \beta_{kl}=\sigma_{ij}(v_l),
\end{equation}
where $i,j$ are the two indices unequal to $l$ and $k$.  Now
$\sigma_{ij}\in\P_2(K)$ is linear on the common edge $e$ of $f_i$
and $f_j$, and, because of the vanishing degrees of freedom of
$\sigma$, $\sigma_{ij}$ is orthogonal to $\P_1$ on $f_i$ and on $f_j$ and
has integral $0$ on $K$.  Therefore, by Lemma~\ref{lemT} applied with $p=\sigma_{ij}$,
it is sufficient to show that $\beta_{lk}$ and $\beta_{kl}$ both vanish
in order to conclude that $\sigma_{ij}$ vanishes.  In fact,
we shall show that the 12 quantities $\beta_{lk}$, corresponding to
the 12 pairs of distinct indices, satisfy a nonsingular homogeneous
system of 12 equations, and so vanish.

The lemma also tells us that
$\sigma_{ij}(v_j)=3(\beta_{lk}+\beta_{kl})/2$.  Interchanging $j$
and $k$ gives
\begin{equation*}
\sigma_{ik}(v_k)=\frac{3}{2}(\beta_{lj}+\beta_{jl}).
\end{equation*}
Also, by definition,
\begin{equation}\label{l}
\beta_{jk} = \sigma_{il}(v_k).
\end{equation}
Combining \eqref{k}--\eqref{l} gives
$$
\sigma_{ij}(v_k) + \sigma_{ik}(v_k) + \sigma_{il}(v_k)
 =\frac32(\beta_{lj}+\beta_{jl})+(\beta_{lk}+\beta_{jk}).
$$
But $\sigma_{ij} + \sigma_{ik} + \sigma_{il}=-\sigma_{ii}$, which
vanishes on $f_i$ and so, in particular, at the vertex $v_k$.
Thus we have established the equation
\begin{equation}\label{e}
a(\beta_{lj}+\beta_{jl})+b(\beta_{lk}+\beta_{jk})=0,
\end{equation}
where $a=3$, $b=2$.

For each of the 12 pairs $(i,k)$ of distinct indices, we let $j$ and $l$ be the remaining indices and
consider the equation \eqref{e}.  In this way we obtain a system of
12 linear equations in
12 unknowns. If we number the pairs of distinct indices lexographically,
the matrix of the system is:
$$
\left(
\begin{array}{cccccccccccc}
 \0 & \0 & \0 & \0 & \0 & \0 & \0 & b & a & \0 & b & a \\
 \0 & \0 & \0 & \0 & b & a & \0 & \0 & \0 & \0 & a & b \\
 \0 & \0 & \0 & \0 & a & b & \0 & a & b & \0 & \0 & \0 \\
 \0 & \0 & \0 & \0 & \0 & \0 & b & \0 & a & b & \0 & a \\
 \0 & b & a & \0 & \0 & \0 & \0 & \0 & \0 & a & \0 & b \\
 \0 & a & b & \0 & \0 & \0 & a & \0 & b & \0 & \0 & \0 \\
 \0 & \0 & \0 & b & \0 & a & \0 & \0 & \0 & b & a & \0 \\
 b & \0 & a & \0 & \0 & \0 & \0 & \0 & \0 & a & b & \0 \\
 a & \0 & b & a & \0 & b & \0 & \0 & \0 & \0 & \0 & \0 \\
 \0 & \0 & \0 & b & a & \0 & b & a & \0 & \0 & \0 & \0 \\
 b & a & \0 & \0 & \0 & \0 & a & b & \0 & \0 & \0 & \0 \\
 a & b & \0 & a & b & \0 & \0 & \0 & \0 & \0 & \0 & \0
\end{array}\right).
$$
Its determinant is $16(2a-b)^2b^6(a+b)^4$, as may be verified with a computer algebra package.
In particular, when $a=3$, $b=2$, the system is nonsingular.                                                              
Thus all the $\beta_{ij}$ vanish as claimed, and the proof is complete.
\end{proof}

Having established unisolvency, the assembled finite element space $\Sigma_h$ is
defined as the set of all matrix fields $\tau$ such that $\tau|_K\in\Sigma_K$ for
all $K\in\T_h$ and for which the degrees of freedom \eqref{dof1} have a 
common value when a face $f$ is shared by two tetrahedra in $\T_h$.  If $\tau\in\Sigma_h$,
then the jump $\jump{\tau n_f}$ of $\tau n_f$ across such an interior face $f$ need not vanish,
but it is orthogonal to $\P_1(f;\R^3)$.  The normal component $\jump{n_f'\tau n_f}$
is, by the definition of the shape function space, linear on each edge of $f$ so
belongs to $\P_1(f)$, and thus
\begin{equation}\label{jump}
\jump{n_f'\tau n_f}=0 \quad\text{on $f$},
\end{equation}
for any interior face of the triangulation.

\section{Error analysis} \label{errors}
In this section, we show that the pair of spaces $\Sigma_h$, $V_h$ give a convergent finite
element method.
The argument follows the one given in \cite{Arnold2003} for the two-dimensional case.
As usual, we suppose that we are given a sequence of tetrahedral meshes $\T_h$
indexed by a parameter $h$ which decreases to zero and represents the maximum
tetrahedron diameter.  We assume that the sequence is shape regular (the ratio
of the diameter of a tetrahedron to the diameter of its inscribed ball is bounded),
and the constants $c$ which appear in the estimates below may depend on this bound,
but are otherwise independent of $h$.

We start by observing that, by construction,
\begin{equation}\label{inclusion}
\div_h \Sigma_h \subset V_h.
\end{equation}
The degrees of freedom determine
an interpolation operator $\Pi_h : H^1 (\Omega;\Sym) \rightarrow \Sigma_h$ by
\begin{align*}
 \int_f (\Pi_h \tau - \tau) n \cdot \ v  \, \ud s & = 0, \quad 
v \in \mathcal{P}_1 (f), \ f\in\Delta_1(\T_h),\\
\int_K  (\Pi_h \tau - \tau )\ud x & = 0, \quad K\in\T_h, 
\end{align*}
where $\Delta_k(\T_h)=\bigcup_{K\in\T_h}\Delta_k(K)$.
Since
$$
\int_K (\div \Pi_h \tau - \div \tau) \cdot v \ud x = - \int_K
(\Pi_h \tau - \tau ) : \epsilon(v) \ud x + \int_{\partial K} (\Pi_h \tau - \tau )
n \cdot v \ud s =0,
$$
for $\tau \in H^1 (K; \Sym )$, $v \in V_K$, $K \in \T_h$,  we
have the commutativity property
\begin{equation}\label{commut}
\div_h \Pi_h \tau = P_h \div\tau,\quad \tau\in H^1(\Omega;\Sym).
\end{equation}
Since $\div$ maps $H^1(\Omega;\Sym)$ onto $L^2(\Omega;\R^3)$, \eqref{commut}
implies that $\div_h$ maps $\Sigma_h$ onto $V_h$.  An immediate consequence
is that the finite element method system \eqref{ncfD} is nonsingular.
Indeed, if $f=0$, then the choice of test functions $\tau=\sigma_h$ and
$v=u_h$ implies that $\sigma_h\equiv 0$ and then, choosing $\tau$ with $\div_h\tau=u_h$,
we get $u_h\equiv 0$.

For the error analysis we also need the approximation and boundedness
properties of the projections $P_h$ and $\Pi_h$.
Obviously, for the $L^2$ projection, we have
\begin{equation}\label{Papprox}
 \|v-P_hv\|_0\le c h^m\|v\|_m,\quad 0\le m\le 2.
\end{equation}
Since $\Pi_h$ is defined
element-by-element
and preserves piecewise linear matrix fields,
we may scale to a reference element of unit diameter using translation, rotation, and
dilation, and use a compactness argument, to obtain
\begin{equation}\label{bdpi0}
\|\tau - \Pi_h \tau\|_0 \leq c h^m \|\tau\|_m, \quad m=1,2,
\end{equation}
where the constant $c$ depends only on the shape regularity of the elements.
See, e.g., \cite{Arnold2002} for details.  Taking $m=1$ and using the triangle
inequality establishes $H^1$ boundedness of $\Pi_h$:
\begin{equation}\label{bdpi}
\| \Pi_h \tau\|_0 \leq c \|\tau\|_1. 
\end{equation}

The final ingredient we need for the convergence analysis is a bound on
the consistency error arising from the nonconformity of the elements.
Define
\begin{equation}\label{defce}
E_h(u,\tau) = \int_\Omega[\eps(u):\tau + \div_h\tau\cdot u]\ud x,\quad 
u\in \ring H^1(\Omega;\R^3),\ \tau\in\Sigma_h +H(\div,\Omega;\Sym).
\end{equation}
If $\tau\in H(\div,\Omega;\Sym)$, then $E_h(u,\tau)=0$, 
by integration by parts.  In general,
\begin{equation*}
E_h(u,\tau) = \sum_{K \in \T_h }
\int_{\partial K} \tau n_K  \cdot u \ud s= \sum_{f \in \Delta_2(\T_h)} \int_f \jump{\tau n_f}\cdot u\ud s,
\end{equation*}
where, again, $\jump{\tau n_f}$ denotes the jump of $\tau n_f$ across the face $f$.
Only the interior faces enter the sum, since $u$ vanishes on
$\partial\Omega$.
Now $\tau n_f =Q_{n_f}(\tau n_f)+(n_f'\tau n_f)n_f$, so
\begin{align*}
E_h(u,\tau)
&= \sum_{f \in \Delta_2(\T_h)} \left\{\int_f \jump{Q_{n_f}(\tau n_f)}\cdot u \ud s
+\int_f \jump{n_f'\tau n_f} ( n_f'u) \ud s\right\}
\\
&= \sum_{f \in \Delta_2(\T_h)} \int_f \jump{Q_{n_f}(\tau n_f)} \cdot u \ud s,
\end{align*}
where the last equality follows from \eqref{jump}.

We let $W_h \subset V_h$ be the 
subspace of the displacement space  $V_h$ consisting of continuous
functions
which are zero on the boundary of $\Omega$. In other words, $W_h$ is
the standard piecewise linear subspace of $\ring H^1(\Omega;\R^3)$.
For any $\tau \in \Sigma_h$ the jumps, $\jump{\tau n_f}$, are orthogonal
to $\P_1(f;\R^3)$, so
$E_h (w, \tau)=0$ for any $w \in W_h$.  

\begin{lemma}\label{consistency}
We may bound the consistency error
\begin{equation}\label{Eerr1}
|E_h (u, \tau)| \leq c h (\|\tau\|_0 + h\| \div_h \tau \|_0)\|u\|_{2} , 
\quad \tau \in \Sigma_h ,\ u \in 
\ring H^1(\Omega;\R^3)\cap H^{2} (\Omega; \R^3).
\end{equation}
Furthermore, for any $\rho \in H^1(\Omega;\Sym)$
\begin{equation}\label{Eerr2}
|E_h (u, \Pi_h\rho)| \leq c h^2 \|\rho \|_1 \|u\|_{2} , 
\quad u \in 
\ring H^1(\Omega;\R^3)\cap H^{2} (\Omega; \R^3).
\end{equation}
\end{lemma}
\begin{proof}
For any $\tau \in \Sigma_h$ we have $E_h(u,  \tau)  =  E_h(u-u_h^I,
\tau)$,
where $u_h^I \in W_h$ is the piecewise linear interpolant
of $u$.
Referring to the definition \eqref{defce}, we obtain
\[
|E_h (u, \tau)| \leq c(\|\div_h\tau \|_0\| u - u_h^I \|_0 + \| 
\tau\|_0 ||\epsilon(u-u_h^I)||_0\\
\leq c h( \|\tau \|_0 + h\|\div_h \tau \|_0) \|u \|_2,
\]
which is \eqref{Eerr1}. For the second estimate we use that
$E_h(u,  \Pi_h\rho)  =  E_h(u - u_h^I, \Pi_h\rho) = E_h(u- u_h^I,  \Pi_h\rho - \rho)$,
which implies that
\[
E_h(u,  \Pi_h\rho)= \sum_{K \in \T_h } \int_{K} \div_h(\Pi_h \rho -
\rho) \cdot (u-u_h^I)\,  dx + \int_{K} (\Pi_h \rho - \rho)
:\epsilon(u-u_h^I) \, dx.
\]
Utilizing the estimate \eqref{bdpi0}, the bound 
\[
|E_h (u, \Pi_h\rho)| \leq c(\|\div \rho \|_0\| u - u_h^I \|_0 + \| \Pi_h \rho -
\rho\|_0 ||\epsilon(u-u_h^I)||_0
\leq c h^2 \|\rho \|_1 \|u \|_2
\]
is an immediate consequence.
\end{proof}

\begin{rem}
The consistency error estimate \eqref{Eerr1} holds for any $u\in\ring H^1(\Omega;\R^3)$
satisfying $u|_K\in H^2(K,\R^3)$ for
each $K \in \T_h$, provided one replaces $\|u\|_{2}$ with the
broken $H^2$ norm $(\sum_{K
\in \T_h } \|u\|_{H^2(K,\R^3)}^2)^{1/2}.$
\end{rem}

With these ingredients assembled,
error bounds for the finite element method now follow in a straightforward fashion.
\begin{thm}
\label{th1}
Let $(\sigma, u)$ be the solution of \eqref{ncf} and $(\sigma_h, u_h)$ the solution of
\eqref{ncfD}. Then
\begin{align}\label{basic-est}
\| \sigma - \sigma_h \|_0 &\leq c h \|u\|_2, \nonumber\\ 
\| \div \sigma - \div_h \sigma_h \|_0 &\leq c h^m \| \div \sigma \|_m, \quad
0 \leq m \leq 2, \\
 \| u - u_h \|_0 &\leq c h\|u\|_{2}.\nonumber
\end{align} 
Furthemore, if problem \eqref{ncf} admits full elliptic regularity, such that
the estimate \eqref{regularity} holds, then
\[
\| u - u_h \|_0 \leq c h^2 \| u \|_2.
\]
\end{thm}
\begin{proof}
Subtracting the first equations of \eqref{ncf} and \eqref{ncfD} and invoking the
definition \eqref{defce} of the consistency error, we get the error equation
\begin{equation}\label{eeq}
 \int_\Omega [A(\sigma-\sigma_h):\tau + (u-u_h)\cdot\div_h\tau]\ud x
 = E_h(u,\tau),\quad \tau\in \Sigma_h.
\end{equation}
Comparing the second equations in \eqref{ncf} and \eqref{ncfD}, we obtain
$\div_h\sigma_h=P_h\div\sigma$, which immediately gives
the claimed error estimate on $\div\sigma$.
Using the commutativity \eqref{commut}, we find that 
$\div_h(\Pi_h\sigma-\sigma_h)=0$.
Choosing $\tau=\Pi_h\sigma-\sigma_h$ in \eqref{eeq}, we get
\begin{equation*}
\int_\Omega A(\sigma-\sigma_h):(\Pi_h\sigma-\sigma_h)\ud x = E_h(u,\Pi_h\sigma-\sigma_h),
\end{equation*}
which implies that
\begin{equation*}
\|\sigma-\sigma_h\|^2_A\le\|\sigma- \Pi_h\sigma\|^2_A
+ 2 E_h(u,\Pi_h\sigma-\sigma_h),
\end{equation*}
where $\|\tau\|_A^2:=\int A\tau:\tau\ud x$. Combining with \eqref{bdpi0} and \eqref{Eerr1}
we conclude that
$$
\|\sigma - \sigma_h\| \le c h ( \|\sigma\|_1 + \|u\|_2) \le  c h\|u\|_2,
$$
which is the desired error estimate for $\sigma$.

To get the error estimate for $u$, we choose $\rho \in H^1(\Omega,\Sym)$ such that
$\div \rho = P_hu - u_h$ and  $\|\rho\|_1 \le c \|P_hu - u_h\|_0$.
Then, in light of the commutativity property \eqref{commut}
and the bound \eqref{bdpi}, $\tau:=\Pi_h\rho\in\Sigma_h$ satisfies
$\div_h\tau=P_hu - u_h$ and $\|\tau\|_0 \le c  \|P_hu - u_h\|_0$.
Hence, using \eqref{inclusion}, \eqref{eeq}, and
\eqref{Eerr1}, we get
\begin{align}\label{uerr-rep}
\|P_hu &- u_h\|^2_0 = \int_{\Omega} \div_h \tau 
\cdot (P_h u - u_h)\, dx = \int_{\Omega} \div_h \tau 
\cdot ( u - u_h)\, dx \nonumber\\
&= - \int_{\Omega} A(\sigma - \sigma_h) :   \tau \, dx + E_h(u,  \tau)
\le c (\|\sigma - \sigma_h\|_0 + h \|u\|_2)\|P_hu - u_h\|_0 .
\end{align}
This gives 
$\|P_hu - u_h\|_0 \le ch \|u\|_2$, and then, by the triangle inequality
and \eqref{Papprox}, the error estimate for $u$.

To establish the final quadratic estimate for $\|u - u_h
\|_0$ in the case of full regularity, we use a duality argument. 
Let $\rho = A^{-1}\eps(w)$, where $w \in \ring
H^1(\Omega;\R^3) \cap H^2(\Omega;\R^3)$ solves the problem $\div
A^{-1}\eps (w) = P_hu - u_h$. It follows from \eqref{regularity} that 
\begin{equation}\label{reg*}
\|\rho \|_1 + \|w \|_2 \le c \|P_h u - u_h \|_0.
\end{equation}
By introducing $w_h^I \in W_h$ as the piecewise linear interpolant of
$w$,
we now obtain from \eqref{uerr-rep} that 
\begin{align*}
\|P_hu - u_h\|^2_0 &= - \int_{\Omega} A(\sigma - \sigma_h) :
\Pi_h\rho \, dx + E_h(u,  \Pi_h\rho)\\
&= - \int_{\Omega} A(\sigma - \sigma_h) :   (\Pi_h\rho-\rho) \, dx +
E_h(u,  \Pi_h\rho)  
- \int_{\Omega} (\sigma - \sigma_h) :\epsilon(w-w_h^I) \, dx,
\end{align*}
where the final equality follows since
\[
\int_{\Omega} (\sigma - \sigma_h) :\epsilon(w_h^I) \, dx = - \sum_{K
  \in \T_h} \int_{K} \div_h (\sigma - \sigma_h) \cdot w_h^I \, dx +
E_h(w_h^I,\sigma- \sigma_h) = 0.
\]
However, by utilizing \eqref{bdpi0}, \eqref{Eerr2}, 
the estimate for $\|\sigma - \sigma_h\|_0$ given in \eqref{basic-est},
combined with the approximation property of the interpolant $w_h^I$,
we obtain from the representation of $\|P_hu - u_h\|^2_0$ above that
\begin{align*}
\|P_hu - u_h\|^2_0 &\le c(h^2\|\rho \|_1\|u \|_2 + \|\sigma -\sigma_h\|
\| \eps(w- w_h^I)\|_0)\\
&\le ch^2 \|u \|_2(\|\rho \|_1 + \|w\|_2) \le c h^2\|u \|_2\|P_hu -
u_h\|_0,
\end{align*}
where we have used \eqref{reg*} to obtain the
final inequality.
This gives $\|P_hu - u_h\|_0 \le ch^2 \| u \|_2$. As above, the
desired estimate for $\| u - u_h \|_0$ now follows from \eqref{Papprox}
and the triangle inequality.
\end{proof}

\begin{rem}
Although $\|\sigma-\Pi_h\sigma\|_0=O(h^2)$,
we  have only shown first order convergence of the finite
element solution: $\|\sigma-\sigma_h\|_0=O(h)$.  The lower rate of convergence
is due to the consistency error estimated in \eqref{Eerr1}.
\end{rem}

\section{The reduced element} \label{reduced}
As for the two-dimensional element in \cite{Arnold2003}, there is a variant of the
element using smaller spaces.  Let
$$
\mathbb{T}(K)= \{\, v\in \P_1(K;\R^3)\,|\, v(x) = a + b\x x, \ a,b\in\R^3\,\},
$$
be the space of rigid motions on $K$.  In the reduced method we take $\tilde V_K:=\mathbb{T}(K)$
instead of $V_K=P_1(K;\R^3)$ as the space of shape functions for displacement, so the
dimension is reduced from 12 to 6.  As shape functions for stress we take
$$
\tilde\Sigma_K = \{ \, \tau \in \Sigma_K\,|\, \div_h \tau \in \mathbb{T}\, \},
$$
so $\dim\tilde\Sigma_K = 36$.
As degrees of freedom for $\tilde{\Sigma}_K$ we take the face moments
\eqref{dof1} but dispense with the interior degrees of freedom \eqref{dof2}.

Let us see how the unisolvence argument adapts to these elements.
If $\tau\in\tilde\Sigma_K$ with vanishing degrees of freedom,
then $\div\tau\in\mathbb T(K)$, and for all $v\in\mathbb T(K)$,
\begin{equation*}
\int_K (\div \tau) v \ud x = - \int_K \tau: \epsilon(v) \ud x + \int_{\partial K} \tau n 
\ v
\ud s = 0,
\end{equation*}
using the degrees of freedom and the fact that $\eps(v)=0$.  Thus $\div\tau=0$ on $K$
and for all $v\in\P_1(K;\R^3)$,
$$
\int_K \tau: \epsilon(v) \ud x = -  \int_K (\div \tau) v \ud x + \int_{\partial K} \tau n 
\ v
\ud s =0.
$$
This shows that $\int_K\tau\ud x=0$, so all degrees of freedom \eqref{dof2} vanish as well.
Therefore the previous unisolvence result applies, and gives $\tau\equiv 0$.

A similar argument establishes the commutativity of the projection into $\tilde\Sigma_h$
(the analogue of \eqref{commut}), and
the analogue of the inclusion \eqref{inclusion} obviously holds. The space $\tilde\Sigma_K$ still contains
$\P_1(K;\Sym)$ so the approximability \eqref{bdpi0} still holds, but the approximability
of $\tilde V_K$ is of one order lower, i.e., in \eqref{Papprox} $m$ can be at most $1$.
As a result, the error estimates given by \eqref{basic-est} in
Theorem~\ref{th1} carry over, 
except that $m$ is limited
to $1$ in the error estimate for $\div\sigma$. 

\bibliographystyle{amsplain}
\bibliography{nonconfelas3d}

\providecommand{\bysame}{\leavevmode\hbox to3em{\hrulefill}\thinspace}
\providecommand{\MR}{\relax\ifhmode\unskip\space\fi MR }
\providecommand{\MRhref}[2]{%
  \href{http://www.ams.org/mathscinet-getitem?mr=#1}{#2}
}
\providecommand{\href}[2]{#2}
\begin{thebibliography}{10}

\bibitem{Adams2005}
Scot Adams and Bernardo Cockburn, \emph{A mixed finite element method for
  elasticity in three dimensions}, J. Sci. Comput. \textbf{25} (2005), no.~3,
  515--521. \MR{2221175 (2006m:65251)}

\bibitem{Amara1979}
Mohamed Amara and Jean-Marie Thomas, \emph{Equilibrium finite elements for the
  linear elastic problem}, Numer. Math. \textbf{33} (1979), no.~4, 367--383.
  \MR{553347 (81b:65096)}

\bibitem{Arnold2005}
Douglas~N. Arnold and Gerard Awanou, \emph{Rectangular mixed finite elements
  for elasticity}, Math. Models Methods Appl. Sci. \textbf{15} (2005), no.~9,
  1417--1429. \MR{2166210 (2006f:65112)}

\bibitem{Arnold2008}
Douglas~N. Arnold, Gerard Awanou, and Ragnar Winther, \emph{Finite elements for
  symmetric tensors in three dimensions}, Math. Comp. \textbf{77} (2008),
  no.~263, 1229--1251. \MR{2398766 (2009b:65291)}

\bibitem{minc}
Douglas~N Arnold and Franco Brezzi, \emph{Mixed and nonconforming finite
  element methods: implementation, postprocessing and error estimates},
  RAIRO-M2AN Modelisation Math et Analyse \textbf{19} (1985), no.~1, 7--32.
  \MR{813687 (87g:65126)}

\bibitem{Arnold1984a}
Douglas~N. Arnold, Franco Brezzi, and Jim Douglas, Jr., \emph{P{EERS}: a new
  mixed finite element for plane elasticity}, Japan J. Appl. Math. \textbf{1}
  (1984), no.~2, 347--367. \MR{840802 (87h:65189)}

\bibitem{Arnold1984}
Douglas~N. Arnold, Jim Douglas, Jr., and Chaitan~P. Gupta, \emph{A family of
  higher order mixed finite element methods for plane elasticity}, Numer. Math.
  \textbf{45} (1984), no.~1, 1--22. \MR{761879 (86a:65112)}

\bibitem{Arnold2006}
Douglas~N. Arnold, Richard~S. Falk, and Ragnar Winther, \emph{Finite element
  exterior calculus, homological techniques, and applications}, Acta Numer.
  \textbf{15} (2006), 1--155. \MR{2269741 (2007j:58002)}

\bibitem{Arnold2007}
\bysame, \emph{Mixed finite element methods for linear elasticity with weakly
  imposed symmetry}, Math. Comp. \textbf{76} (2007), no.~260, 1699--1723
  (electronic). \MR{2336264 (2008k:74057)}

\bibitem{Arnold2002}
Douglas~N. Arnold and Ragnar Winther, \emph{Mixed finite elements for
  elasticity}, Numer. Math. \textbf{92} (2002), no.~3, 401--419. \MR{1930384
  (2003i:65103)}

\bibitem{Arnold2003}
\bysame, \emph{Nonconforming mixed elements for elasticity}, Math. Models
  Methods Appl. Sci. \textbf{13} (2003), no.~3, 295--307, Dedicated to Jim
  Douglas, Jr. on the occasion of his 75th birthday. \MR{1977627 (2004f:65176)}

\bibitem{Awanou2009}
Gerard Awanou, \emph{A rotated nonconforming rectangular mixed element for
  elasticity}, Calcolo \textbf{46} (2009), no.~1, 49--60. \MR{2495247
  (2010c:74057)}

\bibitem{Awanou10b}
\bysame, \emph{Symmetric matrix fields in the finite element method}, Symmetry
  \textbf{2} (2010), no.~3, 1375--1389. \MR{2804836 (2012e:74013)}

\bibitem{Chen2010}
Shao-Chun Chen and Ya-Na Yang, \emph{Conforming rectangular mixed finite
  elements for elasticity}, Journal of Scientific Computing \textbf{47} (2010),
  no.~1, 93--108. \MR{2804836 (2012e:74013)}

\bibitem{Guzman10d}
Bernardo Cockburn, Jayadeep Gopalakrishnan, and Johnny Guzm{\'a}n, \emph{A new
  elasticity element made for enforcing weak stress symmetry}, Math. Comp.
  \textbf{79} (2010), no.~271, 1331--1349. \MR{2629995 (2011m:65276)}

\bibitem{Eisenberg-Malvern}
Martin~A. Eisenberg and Lawrence~E. Malvern, \emph{On finite element
  integration in natural co-ordinates}, International Journal for Numerical
  Methods in Engineering \textbf{7} (1973), no.~4, 574--575.

\bibitem{Falk08}
Richard~S. Falk, \emph{Finite element methods for linear elasticity}, Mixed
  finite elements, compatibility conditions, and applications (Daniele Boffi
  and Lucia Gastaldi, eds.), Lecture Notes in Mathematics, vol. 1939,
  Springer-Verlag, Berlin, 2008, Lectures given at the C.I.M.E. Summer School
  held in Cetraro, June 26--July 1, 2006. \MR{2459075 (2010h:65219)}

\bibitem{Guzman2010}
Jayadeep Gopalakrishnan and Johnny Guzm{\'a}n, \emph{Symmetric nonconforming
  mixed finite elements for linear elasticity}, SIAM J. Numer. Anal.
  \textbf{49} (2011), no.~4, 1504--1520. \MR{2831058}

\bibitem{Guzman2012}
\bysame, \emph{A second elasticity element using the matrix bubble}, IMA J.
  Numer. Anal. \textbf{32} (2012), no.~1, 352--372. \MR{2875255}

\bibitem{Guzman10c}
Johnny Guzm{\'a}n, \emph{A unified analysis of several mixed methods for
  elasticity with weak stress symmetry}, J. Sci. Comput. \textbf{44} (2010),
  no.~2, 156--169. \MR{2659794 (2011h:74021)}

\bibitem{Hu2007/08}
Jun Hu and Zhong-Ci Shi, \emph{Lower order rectangular nonconforming mixed
  finite elements for plane elasticity}, SIAM J. Numer. Anal. \textbf{46}
  (2007/08), no.~1, 88--102. \MR{2377256 (2008m:65321)}

\bibitem{Johnson1978}
Claes Johnson and Bertrand Mercier, \emph{Some equilibrium finite element
  methods for two-dimensional elasticity problems}, Numer. Math. \textbf{30}
  (1978), no.~1, 103--116. \MR{0483904 (58 \#3856)}

\bibitem{Man2009}
Hong-Ying Man, Jun Hu, and Zhong-Ci Shi, \emph{Lower order rectangular
  nonconforming mixed finite element for the three-dimensional elasticity
  problem}, Math. Models Methods Appl. Sci. \textbf{19} (2009), no.~1, 51--65.
  \MR{2484491 (2009m:74012)}

\bibitem{Morley1989}
Mary~E. Morley, \emph{A family of mixed finite elements for linear elasticity},
  Numer. Math. \textbf{55} (1989), no.~6, 633--666. \MR{1005064 (90f:73006)}

\bibitem{Nicaise}
Serge Nicaise, \emph{Regularity of the solutions of elliptic systems in
  polyhedral domains}, Bull. Belg. Math. Soc. Simon Stevin \textbf{4} (1997),
  no.~3, 411--429. \MR{1457079 (98k:35044)}

\bibitem{Stenberg1986}
Rolf Stenberg, \emph{On the construction of optimal mixed finite element
  methods for the linear elasticity problem}, Numer. Math. \textbf{48} (1986),
  no.~4, 447--462. \MR{834332 (87i:73062)}

\bibitem{Stenberg1988a}
\bysame, \emph{A family of mixed finite elements for the elasticity problem},
  Numer. Math. \textbf{53} (1988), no.~5, 513--538. \MR{954768 (89h:65192)}

\bibitem{Stenberg1988}
\bysame, \emph{Two low-order mixed methods for the elasticity problem}, The
  mathematics of finite elements and applications, {VI} ({U}xbridge, 1987),
  Academic Press, London, 1988, pp.~271--280. \MR{956898 (89j:73074)}

\bibitem{Yi2005}
Son-Young Yi, \emph{Nonconforming mixed finite element methods for linear
  elasticity using rectangular elements in two and three dimensions}, Calcolo
  \textbf{42} (2005), no.~2, 115--133. \MR{2158594 (2006h:74071)}

\bibitem{Yi2006}
\bysame, \emph{A new nonconforming mixed finite element method for linear
  elasticity}, Math. Models Methods Appl. Sci. \textbf{16} (2006), no.~7,
  979--999. \MR{2250030 (2007e:65127)}

\end{thebibliography}

\end{document}